\PassOptionsToPackage{table}{xcolor}
\documentclass[pdflatex,sn-mathphys-num]{sn-jnl}

\usepackage{graphicx}
\usepackage{amsmath,amssymb}
\usepackage{xcolor}
\usepackage{array}
\usepackage{manyfoot}
\usepackage[shortlabels]{enumitem}
\usepackage{aliascnt}
\usepackage{cleveref}
\let\orcidlogo\relax
\usepackage{orcidlink}

\DeclareNewFootnote{A}[gobble]
\setlist[enumerate]{leftmargin=.5in}
\setlist[itemize]{leftmargin=.5in}

\newcommand{\newsharedtheorem}[2]{%
  \newaliascnt{#1}{thm}%
  \newtheorem{#1}[#1]{#2}%
  \aliascntresetthe{#1}%
}
\theoremstyle{thmstyleone}
\newtheorem{thm}{Theorem}[section]
\newsharedtheorem{prop}{Proposition}
\newsharedtheorem{cor}{Corollary}
\newsharedtheorem{conj}{Conjecture}
\newsharedtheorem{claim}{Claim}
\newsharedtheorem{lem}{Lemma}
\theoremstyle{thmstylethree}
\newsharedtheorem{dfn}{Definition}
\theoremstyle{thmstyletwo}
\newsharedtheorem{rmk}{Remark}
\newsharedtheorem{eg}{Example}
\crefname{lem}{lemma}{lemmas}
\Crefname{lem}{Lemma}{Lemmas}
\crefname{thm}{theorem}{theorems}
\Crefname{thm}{Theorem}{Theorems}

\DeclareMathOperator{\Li}{li}
\definecolor{RamanujanHeader}{HTML}{173B60}
\definecolor{RamanujanBand}{HTML}{E9F2FA}
\newlength{\RamanujanValueWidth}
\newcommand{\RamanujanHeading}[1]{%
  \rule[-6pt]{0pt}{16pt}%
  {\fontsize{9}{10}\selectfont\textcolor{white}{\bfseries\boldmath$#1$}}%
}

\hypersetup{pdftitle={On proving an Inequality of Ramanujan using Explicit Order Estimates for the Mertens Function},pdfauthor={Subham De}}

\begin{document}

\title[An Inequality of Ramanujan]{On proving an Inequality of Ramanujan using Explicit Order Estimates for the Mertens Function}
\author*[1]{\fnm{Subham} \sur{De}\textsuperscript{\orcidlink{0009-0001-3265-4354}}}\email{subham581994@gmail.com}
\affil*[1]{\orgdiv{Department of Mathematics},
\orgname{Indian Institute of Technology Delhi}, \orgaddress{\country{India}}
(\href{https://www.sites.google.com/view/subhamde}{Webpage})}

\artnote{Earlier versions appeared as preprints on \href{https://doi.org/10.48550/arXiv.2407.12052}{arXiv (2407.12052)}
and \href{https://www.preprints.org/manuscript/202407.0759}{Preprints.org (202407.0759)}.}
\miscnote{\textbf{Funding:} This research received no specific grant from any funding agency in the public, commercial, or not-for-profit sectors.}

\abstract{
We intend to prove unconditionally that the Ramanujan's Inequality,
\begin{align*}
    \pi(x)^2<\frac{ex}{\log x}\pi(x/e)
\end{align*}
for every real $x\geq e^{490}$, where $\pi(x)$ counts the primes up to \(x\). We try to obtain an order estimate for \(\psi(x)-x\) explicitly based on a point-wise power bound on the Mertens function \(M(x)\), where \(\psi\) is the second Chebyshev function. An asymmetric hyperbola decomposition of the
logarithmically weighted Mertens sum and partial summation give an
exact cancellation at the split point. Using the zero-free half-plane,
\(\Re s>7/8\) for the Riemann zeta function, we prove
\(|M(t)|\leq12000\,t^{213/235}\) for \(t\geq e^{131}\) by averaged
Perron inversion. The resulting bounds for \(\psi\) and the prime
counting function \(\pi\) gives us our desired result. A separate argument
using a classical zero-free region proves the asymptotic inequality.
}

\keywords{Riemann Zeta Function, Mertens Function, Chebyshev Function, Arithmetic Functions, Error Estimates, Perron's Formula, M\"obius Inversion, Partial Summation.}

\pacs[AMS subject classifications (2020)]{Primary  11A41,  11A25,  11N05,  11N37,  11N56,  Secondary 11M06,  11M26}

\maketitle

\section{Introduction and Motivation}
Ramanujan's letters to G. H. Hardy in January and February 1913 contain several assertions about the distribution of primes and the prime-counting function $\pi(x)$ \cite[pp.~xxiii--xxx and 349--353]{23}. Hardy later examined some of Ramanujan's formulas in the theory of primes \cite{24}\cite[pp.~234--238]{25}. Ramanujan also recorded the following inequality in his notebooks \cite[p.~112]{1}.

\begin{thm}[Ramanujan's inequality]\label{thm6}
For all sufficiently large real $x$,
\begin{equation}\label{24}
  \pi(x)^2<\frac{ex}{\log x}\pi\left(\frac{x}{e}\right).
\end{equation}
\end{thm}

\Cref{prop:mertens-transfer} gives a quantitative way to pass from a
pointwise estimate for \(M\) to one for \(\psi\). The identity
\(\psi(x)=\sum_{n\leq x}M(x/n)\log n\) follows from classical M\"obius
inversion. Split the associated double sum at \(DN=x\). The contribution
\(M(D)(N\log N-N)\) from partial summation cancels the corresponding
part of the hyperbola intersection term. The remaining boundary term
is bounded by \(|M(D)|(1+\tfrac12\log N)\). For any input bound
\(|M(u)|\leq K u^{1-\eta}\) with \(0<\eta<1\), the resulting estimate
gives \(\psi(x)-x=O_{K,\eta}(x^{1/(1+\eta)}\log x)\).
\Cref{thm:explicit-ramanujan} applies the explicit form of the estimate
to prove \eqref{24} for every real \(x\geq e^{490}\).

The explicit problem is to find a value $x_0$ for which \eqref{24} holds at every real $x\geq x_0$. Wheeler, Keiper, and Galway investigated this problem using Mathematica. Galway found that $38\,358\,837\,677$ is the largest prime counterexample below $10^{11}$ \cite[p.~113]{1}. Under the Riemann hypothesis, Hassani \cite[Theorem~1.2]{22} showed that the inequality holds for $x>138\,766\,146\,692\,471\,228$. Dudek and Platt \cite[Theorem~1.3]{2} showed that $38\,358\,837\,682$ is the largest integer counterexample under the same hypothesis. Dudek and Platt also reported an unconditional cutoff of $\exp(9658)$ and noted that the zero-free region of Mossinghoff and Trudgian \cite{27} yields $\exp(9394)$ within their argument. Axler \cite[Section~1 and Theorem~1.1]{26} discusses a gap in that argument and reports the explicit cutoff $\exp(3158.442)$. Johnston \cite[Theorem~5.1]{johnston} proves the inequality on the finite interval from $38\,358\,837\,683$ to $\exp(103)$.

Barez reports in a preprint that \eqref{24} holds for every real
$x\geq38\,358\,837\,683$ \cite[Corollary~F]{barez}. To cover one
of the intervals in that range, he uses the half-plane $\Re s>7/8$
and an explicit formula for the Chebyshev function.

Costa Pereira \cite[p.~327]{12} combined elementary Chebyshev
estimates with computation to bound the relative error
\(|\psi(x)-x|/x\) by \(1/2976\) above \(10^{11}\). Proposition
\ref{prop:mertens-transfer} starts with the decay rate of \(M(u)/u\)
and gives the resulting decay rate and constants for
\((\psi(x)-x)/x\) as functions of the input parameters. In
\cref{sec:effective-mertens}, the zero-free half-plane
\(\Re s>7/8\) supplies the particular Mertens estimate used here.
Lee and Leong \cite{lee-leong-mertens} also use bounds for the
reciprocal of \(\zeta\) to derive explicit Mertens estimates.
Chirre and Helfgott \cite{chirre-helfgott-bounded} give a
finite-height estimate whose leading term is proportional to \(x\).
The proof in \cref{sec:explicit} carries the resulting Mertens bound
through \(\psi\) to \(\pi\), then compares the errors at \(x\)
and \(x/e\) with the logarithmic-integral margin in \eqref{24}.

\section{Arithmetic Functions}
An arithmetic function is a complex-valued function on the positive integers. We first fix the notation used in the estimates.

\begin{dfn}\label{def2}
Suppose $g(x)$ is nonzero for all sufficiently large $x$. We write $f(x)\sim g(x)$ as $x\to\infty$ if
\[
 \lim_{x\to\infty}\frac{f(x)}{g(x)}=1.
\]
\end{dfn}

\begin{dfn}\label{def4}
For $g(x)>0$ on a tail of the real line, $f(x)=O(g(x))$ means that $|f(x)|\leq Cg(x)$ for some $C>0$ and all sufficiently large $x$. We also write $f(x)\ll g(x)$ for the same bound.
\end{dfn}

\subsection{Prime Counting Function}
\begin{dfn}\label{def1}
For each real $x\geq0$, let $\pi(x)$ be the number of primes $p\leq x$.
\end{dfn}

\begin{thm}[Prime number theorem]\label{thm4}
As $x\to\infty$,
\begin{equation}
 \pi(x)\sim\frac{x}{\log x}.
\end{equation}
Equivalently, $\pi(x)\log x/x\to 1$.
\end{thm}
See \cite{8,18} for accounts of the prime number theorem and \cite[Section~4.10]{17} for a sketch of the elementary proof.

\subsection{Chebyshev Function}
\begin{dfn}\label{def3}
For real $x\geq0$ the second Chebyshev function is $\psi(x)=\sum_{n\leq x}\Lambda(n)$, where the von Mangoldt function is
\begin{equation}\label{22}
 \Lambda(n)=
 \begin{cases}
  \log p, & \text{if }n=p^m\text{ for a prime }p\text{ and an integer }m\geq1,\\
  0, & \text{otherwise}.
 \end{cases}
\end{equation}
\end{dfn}
Grouping the nonzero terms by the exponent of the prime power gives
\begin{equation}\label{23}
 \psi(x)=\sum_{n\leq x}\Lambda(n)
 =\sum_{m=1}^{\infty}\ \sum_{\substack{p\ \mathrm{prime}\\p^m\leq x}}\Lambda(p^m)
 =\sum_{m=1}^{\infty}\ \sum_{\substack{p\ \mathrm{prime}\\p\leq x^{1/m}}}\log p .
\end{equation}
For $x\geq2$, the inner sum is empty when $m>\log x/\log 2$. The asymptotic relation $\psi(x)\sim x$ is equivalent to the prime number theorem \cite{8,18}. For explicit bounds on Chebyshev functions, see \cite{7}.

\subsection{M\"obius Function}
\begin{dfn}\label{def5}
The M\"obius function $\mu\colon\mathbb{N}\to\{0,\pm1\}$ is defined by
\[
 \mu(n)=
 \begin{cases}
  1, & n=1,\\
  (-1)^k, & n=p_1\cdots p_k\text{ for distinct primes }p_1,\ldots,p_k,\quad k\geq1,\\
  0, & p^2\mid n\text{ for some prime }p .
 \end{cases}
\]
Thus the second case applies precisely to squarefree integers $n>1$.
\end{dfn}

\begin{prop}\label{prop2}
For every positive integer $n$ \cite[Theorem~2.1]{17},
\begin{equation}\label{eq:mu-divisor}
 \sum_{d\mid n}\mu(d)=
 \begin{cases}
  1, & n=1,\\
  0, & n>1.
 \end{cases}
\end{equation}
\end{prop}

\subsection{Mertens Function}\label{sec:mertens}
\begin{dfn}\label{def6}
The Mertens function is defined for positive integers $n$ by
\begin{equation}
 M(n)=\sum_{d=1}^{n}\mu(d).
\end{equation}
\end{dfn}
\begin{rmk}\label{rmk1}
For real $x$, we use the extension $M(x)=\sum_{1\leq n\leq x}\mu(n)$, with the empty sum equal to zero.
\end{rmk}

Mertens \cite{4} conjectured that $|M(n)|<\sqrt n$ for all integers $n>1$, a bound consistent with his computations up to $10^4$. Odlyzko and te Riele \cite{odlyzko-riele} disproved this conjecture. Von Sterneck \cite{5} computed $M(n)$ up to $5\cdot10^6$ and proposed the stronger bound $|M(n)|<\tfrac12\sqrt n$ for integers $n>200$. For example, $M(7\,766\,842\,813)=50\,286>\tfrac12\sqrt{7\,766\,842\,813}$ \cite[Table~2]{16}.

The relation between $M(x)$ and the zeros of $\zeta(s)$ comes from the Dirichlet series for $1/\zeta(s)$. References \cite{9,11,12,13,16} discuss estimates and computations for $M(x)$.

\subsection{Some Necessary Derivations}
\begin{prop}\label{prop3}
The Dirichlet series for the M\"obius function satisfies \cite[Chapter~11]{17}
\begin{equation}\label{1}
 \sum_{n=1}^{\infty}\frac{\mu(n)}{n^s}=\frac{1}{\zeta(s)},\qquad \Re(s)>1.
\end{equation}
\end{prop}

\begin{prop}\label{prop1}
As $x\to\infty$, the following estimates hold unconditionally.
\begin{enumerate}[$(i)$]
 \item $\displaystyle \sum_{d\leq x}\frac{\mu(d)}{d}=O\left(\frac{1}{\log x}\right)$.
 \item $\displaystyle \sum_{d\leq x}\frac{\mu(d)\log d}{d}=-1+O\bigl(e^{-c_1\sqrt{\log x}}\bigr)$ for some $c_1>0$.
 \item $\displaystyle M(x)=O\bigl(xe^{-c_0\sqrt{\log x}}\bigr)$ for some $c_0>0$.
\end{enumerate}
In part $(i)$ one also has $O\bigl(e^{-c_1\sqrt{\log x}}\bigr)$ after choosing $c_1$ sufficiently small.
\end{prop}

\begin{proof}
We prove $(iii)$ first. For large real $x$, set $X=\lfloor x\rfloor+\tfrac12$ and $c=1+1/\log X$, so $M(X)=M(x)$. The truncated Perron formula applied to \eqref{1} at this half-integer gives, for $2\leq T\leq X$,
\begin{equation}\label{3}
 M(X)=\frac{1}{2\pi i}\int_{c-iT}^{c+iT}\frac{X^s}{s\zeta(s)}\,ds
       +O\left(\frac{X\log X}{T}\right).
\end{equation}
Indeed, the truncation error is bounded by a constant times
\[
 \sum_{n\geq1}|\mu(n)|\left(\frac{X}{n}\right)^c
   \min\left(1,\frac{1}{T|\log(X/n)|}\right)
\]
\cite[Section~3.12]{3}. For $n$ outside $[X/2,2X]$, this is $O(X^c\zeta(c)/T)=O(X\log X/T)$. On $[X/2,2X]$ we have $|\log(X/n)|\gg |X-n|/X$, and $\sum_{X/2\leq n\leq2X}|X-n|^{-1}=O(\log X)$ because $X$ is a half-integer. This gives $O(X\log X/T)$ for the remaining terms as well.

The classical zero-free region and the reciprocal bound in a narrower region \cite[Sections~3.8 and 3.11]{3} give an $a>0$ such that $1/\zeta(s)$ is holomorphic and $|1/\zeta(s)|\ll\log(3+|t|)$ when $s=\sigma+it$ and $\sigma\geq1-a/\log(3+|t|)$. For bounded $|t|$ this follows, after decreasing $a$, from compactness and nonvanishing on $\Re(s)=1$, with $1/\zeta$ extended through its zero at $s=1$. Put $b=1-a/\log(T+3)$, decreasing $a$ if needed so that $b>1/2$. Cauchy's theorem \cite[Chapter~5]{21} shifts \eqref{3} to the vertical line $\Re(s)=b$ and the two horizontal segments at height $\pm T$. The integrand is holomorphic within this contour. The vertical integral is
\[
 O\left(X^b\int_{-T}^{T}\frac{\log(3+|t|)}{|b+it|}\,dt\right)
   =O\bigl(X^b\log^2(T+3)\bigr).
\]
The two horizontal integrals together are $O(X\log(T+3)/T)$, since $X^\sigma\leq X^c=eX$ and the segments have bounded length. With $T=\exp(\sqrt{\log X})$, the two contour bounds and the Perron error are each $O(Xe^{-c_0\sqrt{\log X}})$ for some $c_0>0$. As $|X-x|\leq\tfrac12$, this proves $(iii)$.

For $(i)$ and $(ii)$ let $f_j(t)=(\log t)^j/t$ for $j=0,1$. By partial summation and $(iii)$,
\[
 \sum_{n>x}\mu(n)f_j(n)=-M(x)f_j(x)-\int_x^\infty M(t)f_j'(t)\,dt .
\]
The vanishing of $M(t)f_j(t)$ as $t\to\infty$ and the absolute convergence of the integral justify taking the upper endpoint to infinity. The integral from $x$ to infinity is bounded by a constant times
\[
 \int_x^\infty \frac{1+(\log t)^j}{t} e^{-c_0\sqrt{\log t}}\,dt
       =O\bigl(e^{-c_1\sqrt{\log x}}\bigr)
\]
for $0<c_1<c_0$, as seen by substituting $u=\sqrt{\log t}$. The term $M(x)f_j(x)$ satisfies the same bound. To identify the sums, write $F(s)=1/\zeta(s)$. For $\Re(s)>1$, differentiation of \eqref{1} gives $\sum_{n\geq1}\mu(n)\log n/n^s=-F'(s)$. The partial-summation integrals for $t^{-s}$ and $(\log t)t^{-s}$ are dominated, for real $s\in[1,2]$, by a constant times $(1+\log t)e^{-c_0\sqrt{\log t}}/t$. Thus their values at $s=1$ are the limits from the right. The simple pole of $\zeta(s)$ with residue $1$ gives $F(1)=0$ and $F'(1)=1$. Consequently the two sums at $s=1$ equal $0$ and $-1$, respectively. The tail bounds give $(i)$ and $(ii)$.
\end{proof}

The estimate $M(x)=O_\varepsilon(x^{1/2+\varepsilon})$ for every $\varepsilon>0$ is equivalent to the Riemann hypothesis \cite[Theorem~14.25(C)]{3}. The stronger bound $M(x)=O(\sqrt{x})$ would in particular imply that hypothesis.

\section{\texorpdfstring{Order Estimates involving $M(x)$}{Order Estimates involving M(x)}}
\subsection{\texorpdfstring{Relation between $\psi(x)$ and $M(x)$}{Relation between psi(x) and M(x)}}
\begin{thm}\label{thm1}
For every real $x\geq1$,
\begin{align}
 \sum_{n\leq x}M(x/n)&=1,\label{eq:unweighted}\\
 \psi(x)&=\sum_{n\leq x}M(x/n)\log n.\label{4}
\end{align}
\end{thm}

\begin{proof}
Expanding the Mertens function and grouping the pairs $(d,n)$ by $k=dn$ gives
\[
 \sum_{n\leq x}M(x/n)=\sum_{dn\leq x}\mu(d)
 =\sum_{k\leq x}\sum_{d\mid k}\mu(d)=1
\]
by \eqref{eq:mu-divisor}. The identity \(\log k=\sum_{d\mid k}\Lambda(d)\) gives, by M\"obius inversion \cite[Chapter~2]{17},
\[
 \Lambda(k)=\sum_{d\mid k}\mu(d)\log(k/d).
\]
Summing this exact identity over $k\leq x$ gives \eqref{4}.
\end{proof}

\subsection{\texorpdfstring{An approximation for $\psi(x)$ and its relation to $\pi(x)$}{An approximation for psi(x) and its relation to pi(x)}}
\begin{lem}\label{lemma1}
There is an absolute constant $c_2>0$ such that, as $x\to\infty$,
\begin{equation}\label{10}
 \psi(x)=x+O\bigl(xe^{-c_2\sqrt{\log x}}\bigr).
\end{equation}
\end{lem}
This is the classical de la Vall\'ee Poussin estimate \cite[Chapter~III]{3}. The following proof derives it from the estimates for $M(x)$.
\begin{proof}
To estimate \eqref{4}, put $y=\sqrt{x}$ and $L_0(u)=\sum_{n\leq u}\log n$ for real $u\geq1$. Every pair with $dn\leq x$ has $d\leq y$ or $n\leq y$, and pairs with both $d,n\leq y$ all satisfy $dn\leq x$. It follows that
\begin{equation}\label{eq:hyperbola}
 \psi(x)=\sum_{d\leq y}\mu(d)L_0(x/d)
         +\sum_{n\leq y}\log n\,M(x/n)-M(y)L_0(y).
\end{equation}
By integral comparison, $L_0(u)=u\log u-u+O(\log(2u))$. The first sum in \eqref{eq:hyperbola} therefore equals
\[
 x(\log x-1)\sum_{d\leq y}\frac{\mu(d)}{d}
 -x\sum_{d\leq y}\frac{\mu(d)\log d}{d}
 +O(\sqrt{x}\log x).
\]
By \Cref{prop1}, this is $x+O(x\log x\,e^{-c_1\sqrt{\log y}}+\sqrt{x}\log x)$. For $n\leq y$ one has $x/n\geq y$. The second sum in \eqref{eq:hyperbola} therefore satisfies
\[
 \left|\sum_{n\leq y}\log n\,M(x/n)\right|
 \ll xe^{-c_0\sqrt{\log y}}\sum_{n\leq y}\frac{\log n}{n}
 \ll x(\log x)^2e^{-c_0\sqrt{\log y}} .
\]
Also $M(y)L_0(y)=O(x\log x\,e^{-c_0\sqrt{\log y}})$. Since $\log y=\tfrac12\log x$, any fixed $c_2$ with $0<c_2<\min(c_0,c_1)/\sqrt{2}$ absorbs the powers of $\log x$ and gives \eqref{10}.
\end{proof}

In view of \eqref{eq:unweighted}, an estimate of the form $\psi(x)=x-\sum_{n\leq x}M(x/n)+O(\sqrt{x})$ would give $\psi(x)=x+O(\sqrt{x})$, contrary to the oscillation bound $\psi(x)-x=\Omega_\pm(\sqrt{x}\log\log\log x)$ \cite[Chapter~V]{8}.

\begin{thm}\label{thm2}
For real $x\geq2$, if $\theta(x)=\sum_{p\leq x}\log p$, then
\begin{equation}\label{eq:theta-abel}
 \theta(x)=\pi(x)\log x-\int_2^x\frac{\pi(t)}{t}\,dt.
\end{equation}
In particular, as $x\to\infty$,
\begin{equation}\label{9}
 \pi(x)=\frac{\psi(x)}{\log x}+O\left(\frac{x}{\log^2 x}\right).
\end{equation}
\end{thm}
\begin{proof}
For each prime $p\leq x$, subtract $\int_p^x dt/t$ from $\log x$ to recover $\log p$. Summing over $p$ and interchanging the finite sum and integral gives the Abel summation identity \eqref{eq:theta-abel}. The bound $\pi(t)\ll t/\log t$ follows from \Cref{thm4}, after enlarging the constant to cover $t\geq2$. For $x\geq4$, splitting the integral at $\sqrt{x}$ gives
\[
 \int_2^x\frac{\pi(t)}{t}\,dt
 \ll \frac{\sqrt{x}}{\log 2}+\frac{2x}{\log x}
 =O(x/\log x).
\]
In fact, the prime number theorem and $\int_2^x dt/\log t\sim x/\log x$ give $\int_2^x\pi(t)\,dt/t\sim x/\log x$ by comparison on the tail of the integral.

Only primes $p\leq\sqrt{x}$ contribute higher powers to $\psi(x)$. For each such prime
the number of contributing powers with exponent at least $2$ is at most $\log x/\log p$. Thus \eqref{23} gives
\[
 0\leq\psi(x)-\theta(x)\leq\pi(\sqrt{x})\log x\leq\sqrt{x}\log x.
\]
Dividing \eqref{eq:theta-abel} by $\log x$ now gives
\[
 \pi(x)=\frac{\psi(x)}{\log x}
          +O(x/\log^2 x)+O(\sqrt{x}).
\]
The function $u^2e^{-u/2}$ decreases for $u\geq9$, and $81e^{-9/2}<1$. Hence $\sqrt{x}\leq x/\log^2 x$ for $x\geq e^9$, which yields \eqref{9}.
\end{proof}

\subsection{\texorpdfstring{An asymptotic estimate for $\pi(x)$}{An asymptotic estimate for pi(x)}}
The single-term approximation in \eqref{9} does not retain enough terms to determine the sign in Ramanujan's inequality. We will use the estimate for $\psi(x)$ to keep five terms in the expansion of $\pi(x)$.

\begin{thm}\label{thm3}
Write $\Li(x)=\operatorname{PV}\int_0^x dt/\log t$ for $x>1$. As $x\to\infty$, for some $c_3>0$,
\begin{align}
 \pi(x)&=\Li(x)+O\bigl(xe^{-c_3\sqrt{\log x}}\bigr),\label{eq:pi-li}\\
 \pi(x)&=x\sum_{k=0}^{4}\frac{k!}{(\log x)^{k+1}}
           +O\left(\frac{x}{(\log x)^6}\right).\label{15}
\end{align}
\end{thm}
\begin{proof}
By \eqref{10} and the prime-power bound in the proof of \Cref{thm2},
$\theta(t)=t+O(te^{-c_3\sqrt{\log t}})$ for some $0<c_3<c_2$. Partial summation in the other direction yields
\begin{equation}\label{eq:pi-theta}
 \pi(x)=\frac{\theta(x)}{\log x}+\int_2^x\frac{\theta(t)}{t\log^2 t}\,dt.
\end{equation}
The main terms are $x/\log x+\int_2^x dt/\log^2 t=\Li(x)+2/\log 2-\Li(2)$. In the integral of the error, the range $[2,\sqrt{x}]$ contributes $O(\sqrt{x})$. On $[\sqrt{x},x]$ one has $\log t\geq(\log x)/2$, so this range contributes $O(xe^{-(c_3/\sqrt{2})\sqrt{\log x}})$. The endpoint error is $O(xe^{-c_3\sqrt{\log x}}/\log x)$. Relabeling the positive constant in the exponential gives \eqref{eq:pi-li}. Repeated integration by parts in $\int_2^x dt/\log t$ gives the five terms in \eqref{15} and a remainder bounded by a constant plus $120\int_2^x dt/\log^6 t$. A split at $\sqrt{x}$ bounds the latter integral by $O(x/\log^6 x)$. The exponential error in \eqref{eq:pi-li} is also $O(x/\log^6 x)$.
\end{proof}

\section{Proving Ramanujan's Inequality}
\subsection{\texorpdfstring{An asymptotic expression for $\mathcal{G}(x)$}{An asymptotic expression for G(x)}}
For $x>1$ define
\begin{equation}\label{27}
 \mathcal{G}(x)=\pi(x)^2-\frac{ex}{\log x}\pi\left(\frac{x}{e}\right).
\end{equation}
The leading terms of size $x^2/\log^2 x$ cancel. To find the first term that survives, let $L=\log x$. From \eqref{15},
\begin{align}\label{17}
 \frac{\pi(x)^2}{x^2}
 &=\frac{1}{L^2}+\frac{2}{L^3}+\frac{5}{L^4}
     +\frac{16}{L^5}+\frac{64}{L^6}+O(L^{-7}).
\end{align}
For the other term in \eqref{27}, use \eqref{15} at $x/e$ and expand $(L-1)^{-k-1}$ in powers of $L^{-1}$. This gives
\begin{align}
 \frac{e}{xL}\pi(x/e)
 &=\frac{1}{L}\sum_{k=0}^4 \frac{k!}{(L-1)^{k+1}}+O(L^{-7})\notag\\
 &=\frac{1}{L^2}+\frac{2}{L^3}+\frac{5}{L^4}
      +\frac{16}{L^5}+\frac{65}{L^6}+O(L^{-7}).\label{16}
\end{align}
The error in \eqref{17} is $O(L^{-7})$ because the error in $\pi(x)/x$ is $O(L^{-6})$ and $\pi(x)/x=O(L^{-1})$. For \eqref{16}, the factor $L^{-1}$ multiplies an error of order $(L-1)^{-6}$.

\subsection{Sign for sufficiently large \texorpdfstring{$x$}{x}}\label{sec:sign}
Subtracting \eqref{16} from \eqref{17}, we obtain
\begin{equation}\label{21}
 \mathcal{G}(x)=-\frac{x^2}{(\log x)^6}
                 +O\left(\frac{x^2}{(\log x)^7}\right).
\end{equation}
After division by $x^2/(\log x)^6$, the right-hand side is $-1+O(1/\log x)$, so $\mathcal{G}(x)<0$ for all sufficiently large $x$. This proves \Cref{thm6}.

By \eqref{eq:unweighted}, $\sum_{n\leq x}M(x/n)=1$. Its contribution
$-2x/L^2$ in the earlier version is absorbed by an error of order $x^2/L^2$.
The expansion in \eqref{21} resolves the sign on the smaller scale
$x^2/L^6$.

\section{Explicit estimates from the Mertens function}
The quantitative estimate in \cref{prop:mertens-transfer} accepts any
pointwise power bound for \(M\). We first obtain such a bound with
explicit constants and then apply the proposition.

\subsection{Asymptotic prime-counting bounds}
Set \(S_4(x)=x\sum_{k=0}^4 k!/(\log x)^{k+1}\). The two-sided bound
\begin{equation}\label{eq:asymptotic-pi-two-sided}
 |\pi(x)-S_4(x)|\leq \frac{Cx}{(\log x)^6},\qquad x\geq x_1,
\end{equation}
follows from \eqref{15} for some \(C>0\) and \(x_1>1\).
An explicit cutoff for Ramanujan's inequality requires numerical
control of the errors at both arguments. We obtain that control from
the bound on \(M\) below.

\subsection{An explicit bound for the Mertens function}\label{sec:effective-mertens}

\begin{thm}[Zero-free half-plane {\cite[Theorem~1.1]{qriemann}}]\label{thm:qriemann}
The Riemann zeta function has no zero in \(\Re s>7/8\).
Every nontrivial zero \(\rho=\beta+i\gamma\) lies in
\(1/8\leq\beta\leq7/8\).
\end{thm}

\begin{proof}[Deduction for \(\zeta\) from the Hecke theorem]
The preprint proves nonvanishing on this half-plane for primitive
finite-order Hecke \(L\)-functions over \(F=\mathbb Q(\sqrt{-3})\).
Take the trivial character of conductor one, whose \(L\)-function is
\(\zeta_F\). The quadratic character \(\chi_{-3}\) has values \(1\)
at \(p\equiv1\pmod 3\), \(-1\) at \(p\equiv2\pmod 3\), and \(0\) at
\(3\). The corresponding primes split, remain inert, and ramify,
respectively, in \(F\). In fact \(F\) is generated by a primitive cube
root of unity, so the splitting at \(p\ne3\) follows from the existence
of a root of order three modulo \(p\). The prime \(3\) divides the
discriminant \(-3\).

At a split prime the factor of \(\zeta_F\) is \((1-p^{-s})^{-2}\).
At an inert prime it is \((1-p^{-2s})^{-1}\), the product of
\((1-p^{-s})^{-1}\) and \((1+p^{-s})^{-1}\). At \(3\) the factor is
\((1-3^{-s})^{-1}\). Thus on \(\Re s>1\), and then by meromorphic
continuation, the Euler products give
\begin{equation}\label{eq:zeta-quadratic-transfer}
 \zeta_F(s)=\zeta(s)L(s,\chi_{-3}).
\end{equation}
The partial sums of \(\chi_{-3}\) are bounded. By partial summation,
the Dirichlet series for \(L(s,\chi_{-3})\) converges locally
uniformly on \(\Re s>0\), where it defines a holomorphic function. On \(\Re s>7/8\), a zero of \(\zeta\) would therefore
give a zero of \(\zeta_F\). At \(s=1\) the functions \(\zeta\) and
\(\zeta_F\) have their principal poles. The functional equation
reflects the real parts of the nontrivial zeros about \(1/2\),
giving the lower bound \(\beta\geq1/8\).
\end{proof}

The written Hecke theorem appears in the preprint \cite{qriemann}.
For the zeta statement used here, Goldblatt \cite{qriemann-audit}
checked the Lean theorem
\nolinkurl{OAI.riemannZeta_ne_zero_of_seven_eighths_lt_re}
with both the Lean and nanoda kernels at the commit cited in that
reference. Appendix~\ref{app:continuation} gives the Mellin-continuation
argument that converts the preprint's estimates into nonvanishing.
The nonvanishing makes \(1/\zeta(s)\) holomorphic throughout the
Perron strip. At \(s=1\), the simple pole of \(\zeta\) gives a zero
of the reciprocal. The next lemma bounds its growth on the strip.

\begin{lem}\label{lem:reciprocal-mertens}
Put $a=89/100$, $b=27/25$ and $p=7/40$. Then
\begin{equation}\label{eq:reciprocal-vertical}
 |1/\zeta(a+it)|\leq
 \begin{cases}
  400, & |t|\leq 2\pi,\\
  14000\bigl(|t|/(2\pi)\bigr)^p, & |t|\geq 2\pi.
 \end{cases}
\end{equation}
In addition, the second bound holds throughout $a\leq\Re s\leq b$ for
$|\Im s|\geq 2\pi$. The reciprocal is extended holomorphically at $s=1$.
\end{lem}

\begin{proof}
Write $F=1/\zeta$ and $H(s)=(s-1)\zeta(s)$, with $H(1)=1$. By
\Cref{thm:qriemann}, $H$ is nonzero on the closed strip used below. The
Hadamard product for the completed zeta function gives
\begin{equation}\label{eq:zeta-real-derivative}
 \Re\frac{H'}{H}(s)
 = \sum_\rho \frac{\sigma-\beta}{|s-\rho|^2}
   -\Re\frac{1}{s}+\frac{\log\pi}{2}
   -\frac12\Re\frac{\Gamma'}{\Gamma}(s/2),
 \qquad s=\sigma+it .
\end{equation}
Here and below zeros are counted with multiplicity. To see why the sum has
no additional real constant, write $\xi(s)=\tfrac12 sH(s)\pi^{-s/2}\Gamma(s/2)$.
The real part of the Hadamard product's logarithmic derivative has the
form $C+\sum_\rho\Re(1/(s-\rho))$. At $s=1/2$ the sum vanishes by the
symmetry $\rho\mapsto1-\overline\rho$ and the derivative vanishes by
$\xi(s)=\xi(1-s)$, with $\xi(1/2)\ne0$ by the nonvanishing of $\zeta$
on the real interval $(0,1)$. The real-part series converges absolutely
for each fixed $s$ away from the zeros. For these zeta identities, see
\cite[Sections~2.1 and~2.12--2.13]{3}. Put
$S_b(t)=\sum_\rho(b-\beta)/|b+it-\rho|^2$. For $a\leq\sigma\leq b$,
\[
 \frac{\sigma-\beta}{|\sigma+it-\rho|^2}
 \leq\frac{b-\beta}{\sigma-\beta}
       \frac{b-\beta}{|b+it-\rho|^2}
 \leq\frac{b-7/8}{\sigma-7/8}
       \frac{b-\beta}{|b+it-\rho|^2}.
\]
The first inequality follows by multiplying by the two squared
denominators. The second uses that $(b-\beta)/(\sigma-\beta)$ increases
with $\beta\leq7/8$. Hence, for $a\leq\sigma\leq b$,
\begin{equation}\label{eq:zero-integral-mertens}
 \int_\sigma^b \sum_\rho
      \frac{u-\beta}{|u+it-\rho|^2}\,du
 \leq K_\sigma S_b(t),\qquad
 K_\sigma=(b-7/8)\log\frac{b-7/8}{\sigma-7/8}.
\end{equation}

For $|t|\geq2\pi$, set $w=|t|/(2\pi)$. The gamma term is controlled by
\begin{equation}\label{eq:gamma-mertens}
 \left|\Re\frac{\Gamma'}{\Gamma}(s/2)-\log(|t|/2)\right|<0.08
 \qquad(a\leq\sigma\leq b).
\end{equation}
For completeness, write $z=s/2$, $r=\Re z$ and $v=|\Im z|$. Euler
summation of the harmonic series for $\Gamma'/\Gamma$, followed by one
integration by parts, gives
\[
 \frac{\Gamma'}{\Gamma}(z)=\log z-\frac{1}{2z}-\frac{1}{12z^2}
       +\int_0^\infty\frac{B_2(\{y\})}{(z+y)^3}\,dy ,
 \qquad B_2(u)=u^2-u+\tfrac16 .
\]
Indeed, the first Euler summation remainder is
$\int_0^\infty(\{y\}-\tfrac12)/(z+y)^2\,dy$. Its integral by parts
gives the displayed $-1/(12z^2)$ and the $B_2$ integral. Here
$0<r\leq0.54$, $|B_2|\leq1/6$, and
$\int_0^\infty(v^2+y^2)^{-3/2}\,dy=1/v^2$. Thus the difference in
\eqref{eq:gamma-mertens} is at most
$(0.54^2/2+0.54/2+1/12+1/6)/v^2<0.08$ since $v\geq\pi>3$.

On the right boundary, the Euler product bounds
$|F(b+it)|\leq\zeta(b)\leq1+1/(b-1)=13.5$. It also bounds
$\Re(\zeta'/\zeta)(b+it)$ above by $-\zeta'(b)/\zeta(b)$.
The identity
\[
 \zeta(u)=\frac{u}{u-1}-u\int_1^\infty\{y\}y^{-u-1}\,dy
 \qquad(u>1)
\]
gives $-\zeta'(b)\leq (b-1)^{-2}+1/b$ on differentiation. Since
$\zeta(b)\geq 1/(b-1)$ by integral comparison, we have
$-\zeta'(b)/\zeta(b)<12.575$. At $s=b+it$ the real parts of $1/s$
and $1/(s-1)$ sum to at most $(2b-1)/t^2$. Equations
\eqref{eq:zeta-real-derivative} and \eqref{eq:gamma-mertens} now give
\[
 S_b(t)\leq12.73+\tfrac12\log w,\qquad
 \Re(H'/H)(\sigma+it)\leq
 \sum_\rho\frac{\sigma-\beta}{|\sigma+it-\rho|^2}
       -\tfrac12\log w+0.04 .
\]
For $s\ne1$ the exact formula along the horizontal segment is
\begin{equation}\label{eq:log-reciprocal-mertens}
 |F(\sigma+it)|=|F(b+it)|\frac{|\sigma-1+it|}{|b-1+it|}
  \exp\left(\int_\sigma^b\Re\frac{H'}{H}(u+it)\,du\right).
\end{equation}
For $|t|\geq 2\pi$ the ratio in \eqref{eq:log-reciprocal-mertens}
is below $1.01$. In the bound for the exponent, the terms involving
$\log w$ are $\tfrac12K_\sigma\log w$ from the integrated zero sum
and $-\tfrac12(b-\sigma)\log w$ from the gamma term. The power of $w$ is thus
$\tfrac12(K_\sigma-(b-\sigma))$. The expression $K_\sigma-(b-\sigma)$
decreases on $[a,b]$ because its derivative is
$1-(b-7/8)/(\sigma-7/8)\leq0$. Since
$K_a=(41/200)\log(41/3)<0.537$, this power is below $0.1735<p$.
As $w\geq1$, we may use $w^p$ throughout the strip. The coefficient
is bounded by
\[
 14\exp\bigl(12.73\cdot0.537+0.04\cdot0.19\bigr)<14000 .
\]

For $|t|\leq2\pi$, use $|\gamma|>14$ and
$\sum_\rho\gamma^{-2}=2\sum_{\gamma>0}\gamma^{-2}<0.047$
\cite[Example~1]{brent-platt-trudgian}.
Here $|t-\gamma|>0.55|\gamma|$ and $b-\beta\leq b$, so
$S_b(t)<b(0.047)/(0.55)^2<0.2$.
The digamma series gives
$\Re(\Gamma'/\Gamma)(z)\geq(\Gamma'/\Gamma)(\Re z)$ for $\Re z>0$,
by comparing each reciprocal term. For $r\geq0.445$,
$(\Gamma'/\Gamma)(r)\geq-\gamma_{\!E}-1/r>-3.25$, by the recurrence
and the increase of the digamma function on the positive real axis.
Thus the terms in \eqref{eq:zeta-real-derivative} outside the zero sum
are bounded above by $2.3$. In \eqref{eq:log-reciprocal-mertens} the
ratio is at most $11/8$, and
$K_a\cdot0.2+(b-a)\cdot2.3<0.55$. This gives
$|F(a+it)|\leq13.5(11/8)e^{0.55}<400$.
\end{proof}

\begin{thm}[Explicit Mertens estimate]\label{thm:explicit-mertens}
For every real $x\geq e^{131}$,
\begin{equation}\label{eq:explicit-mertens}
 |M(x)|\leq 12000\,x^{213/235}.
\end{equation}
\end{thm}

\begin{proof}
We average $M$ over $[x,x+d]$. The integral estimate below yields
a term of size $x^a(x/d)^p$ for the average, and recovering $M(x)$
costs at most $d/2+1$. Balancing these powers at $d$ proportional to
$x^{1-\eta}$ gives $a+p\eta=1-\eta$, so $\eta=22/235$.
Fix $x\geq e^{131}$ and take $d=2000x^{1-\eta}$.
Since $131\eta>12.25$ and $e^{12.25}>200000$, we have $d/x<0.01$.
The average $\mathcal M_d(x)=d^{-1}\int_x^{x+d}M(t)\,dt$ satisfies
\begin{equation}\label{eq:smoothed-perron-mertens}
 \mathcal M_d(x)=\frac{1}{2\pi i}\int_{a-i\infty}^{a+i\infty}
 F(s)\frac{(x+d)^{s+1}-x^{s+1}}{d\,s(s+1)}\,ds.
\end{equation}
To obtain the formula on $\Re s=b$, insert the absolutely convergent
series for $F(s)$ in the integral with kernel $y^{s+1}/(s(s+1))$.
For $z>1$, the Mellin inversion of $z^s/(s(s+1))$ is $1-1/z$.
For $0<z\leq1$ it is zero. Applying this with $z=y/n$ gives
$\sum_{n\leq y}\mu(n)(y-n)=\int_1^y M(t)\,dt$. The summand vanishes
at $n=y$, so the formula holds at integers as well. Subtract the
two integrals at $y=x+d$ and $y=x$ and divide by $d$. The integrand
is holomorphic across the strip, including at the zero of $F$ at $s=1$.
For fixed $x,d$ the kernel is $O(|\Im s|^{-2})$ uniformly across the
strip. Since $p<1$, \Cref{lem:reciprocal-mertens} ensures absolute
convergence on the vertical lines and vanishing horizontal integrals
in the shift from $b$ to $a$.

On the line $s=a+it$, let $z_0=x+d$ and $U=2z_0/d$. By integrating
$y^s$ over $x\leq y\leq z_0$ we bound the absolute value of the kernel
in \eqref{eq:smoothed-perron-mertens} by $z_0^a/|a+it|$.
Taking absolute values of its two endpoint terms instead gives
$z_0^aU/t^2$ for $t\ne0$. Use the first bound up to $2\pi$ and the
minimum of the two above $2\pi$. Combining with
\eqref{eq:reciprocal-vertical} gives
\[
 |\mathcal M_d(x)|
 \leq \frac{z_0^a}{\pi}
 \left(400\operatorname{arsinh}(2\pi/a)
       +\frac{14000}{p(1-p)}\left(\frac{U}{2\pi}\right)^p\right).
\]
In the second term we extended the majorizing integral to
$0<t<\infty$, integrating $t^{p-1}$ up to $U$ and
$Ut^{p-2}$ above $U$. Since $z_0^a\leq 1.01x^a$,
$U/(2\pi)\leq (1.01/(2000\pi))x^\eta$, and
$\operatorname{arsinh}(2\pi/a)<3$, we obtain
\[
 |\mathcal M_d(x)|\leq
 \frac{1.01}{\pi}
 \left(1200+\frac{14000}{p(1-p)}
        \left(\frac{1.01}{2000\pi}\right)^p\right)x^{1-\eta}.
\]
Here we used $a+p\eta=1-\eta$ and $x^a\leq x^{1-\eta}$.
For $x\leq t\leq x+d$ the number of integers in $(x,t]$ is at most
$t-x+1$, so $|\mathcal M_d(x)-M(x)|\leq d/2+1$.
Using $\pi>3$ and $(1.01/6000)^p<0.28$, the combined coefficient of
$x^{1-\eta}$ is at most
\[
 \frac{1.01}{3}\left(1200+\frac{14000\cdot0.28}{p(1-p)}\right)+1001
 <12000 .
\]
\end{proof}

\subsection{Transferring the bound to the Chebyshev function}\label{sec:effective-transfer}

The error terms in the hyperbola decomposition depend on its two split
points. We first keep the boundary terms together, then choose the split
to balance the errors that remain.

\begin{prop}[A quantitative hyperbola estimate]\label{prop:mertens-transfer}
Let $0<\eta<1$, $K>0$, and $u_0\geq1$. Suppose that
$|M(u)|\leq K u^{1-\eta}$ for all $u\geq u_0$. For $D,N\geq1$
with $DN=x$ and $D\geq u_0$, one has
\begin{equation}\label{eq:effective-hyperbola}
\begin{aligned}
 |\psi(x)-x|\leq{}&
 \frac{D}{2}(3+\log N)
 +K D^{1-\eta}\bigl(1+\tfrac12\log N\bigr)\\
 &+\frac{2Kx}{\eta}D^{-\eta}\log N
 +\frac{K}{\eta^2}x^{1-\eta}.
\end{aligned}
\end{equation}
\end{prop}
\begin{proof}
Write $L_0(u)=\sum_{n\leq u}\log n$. The same partition of pairs
as in \eqref{eq:hyperbola}, with $d\leq D$ and $n\leq N$, gives
\begin{equation}\label{eq:effective-hyperbola-identity}
 \psi(x)=\sum_{d\leq D}\mu(d)L_0(x/d)
         +\sum_{n\leq N}M(x/n)\log n-M(D)L_0(N).
\end{equation}
Define $r(u)=L_0(u)-u\log u+u$. For every $u\geq1$,
\begin{equation}\label{eq:factorial-remainder-mertens}
 |r(u)|\leq \tfrac12\log u+1.
\end{equation}
In fact, for an integer $m\geq1$ the trapezoidal sum for $\log t$ on
$[1,m]$ is $L_0(m)-\tfrac12\log m$. On each $[j,j+1]$ the integral
exceeds the trapezoid by at most $1/(12j^2)$, from the second derivative
bound $|(\log t)''|\leq j^{-2}$. The total error is less than $1/6$ since
$\sum_{j\geq1}j^{-2}<2$. For $m=\lfloor u\rfloor$, subtracting
$\int_m^u\log t\,dt\in[0,\log u]$ gives
$1-1/6+\tfrac12\log m-\log u\leq r(u)\leq 1+\tfrac12\log u$.
The lower bound implies \eqref{eq:factorial-remainder-mertens} because
$u/m<2$ and $\log 2<1$.

Let $q_0(t)=x\bigl(\log(x/t)-1\bigr)/t$. By the convergent sums
established in the proof of \Cref{prop1},
$\sum_{d\geq1}\mu(d)q_0(d)=x$. Partial summation on the tail, with
$q_0'(t)=-x\log(x/t)/t^2$, gives
\[
 \sum_{d\leq D}\mu(d)q_0(d)
 =x+M(D)q_0(D)-x\int_D^\infty\frac{M(t)\log(x/t)}{t^2}\,dt .
\]
The endpoint at infinity vanishes by the assumed bound on $M(t)$.
Since $q_0(D)=N\log N-N$, the endpoint contribution $M(D)q_0(D)$
cancels the $M(D)(N\log N-N)$ part of the intersection term
$-M(D)L_0(N)$, leaving $-M(D)r(N)$.
Substituting $u=x/t$ in the integral gives
\begin{equation}\label{eq:mertens-sum-integral}
\begin{aligned}
 \psi(x)-x={}&\sum_{d\leq D}\mu(d)r(x/d)-M(D)r(N)\\
 &+\sum_{n\leq N}M(x/n)\log n-\int_0^N M(x/u)\log u\,du .
\end{aligned}
\end{equation}
By \eqref{eq:factorial-remainder-mertens}, the remaining boundary
contribution is at most $|M(D)|(1+\tfrac12\log N)$ in absolute value.

For $0<\eta<1$, the last two terms in \eqref{eq:mertens-sum-integral}
are bounded in absolute value by
\[
 Kx^{1-\eta}\log N\sum_{n\leq N}n^{\eta-1},
 \qquad
 Kx^{1-\eta}\left(\frac{1}{\eta^2}
          +\log N\int_1^N u^{\eta-1}\,du\right),
\]
respectively. For the part on $(0,1)$ we used
$\int_0^1 u^{\eta-1}|\log u|\,du=1/\eta^2$.
The monotonicity of $u^{\eta-1}$ gives
$\sum_{n\leq N}n^{\eta-1}\leq 1+(N^\eta-1)/\eta\leq N^\eta/\eta$,
and $x^{1-\eta}N^\eta=xD^{-\eta}$. Finally,
\[
 \sum_{d\leq D} |r(x/d)|
 \leq D\bigl(\tfrac12\log N+1\bigr)
     +\tfrac12\sum_{d\leq D}\log(D/d)
 \leq \frac D2(\log N+3),
\]
where the last sum is at most $\int_0^D\log(D/t)\,dt=D$ by monotonicity.
Together with the bound for $M(D)r(N)$, this proves
\eqref{eq:effective-hyperbola}.
\end{proof}

The proposition already gives a power saving for an arbitrary input.
Put \(\delta=1/(1+\eta)\) and choose \(D=x^\delta\),
\(N=x^{1-\delta}\) for \(x\geq u_0^{1+\eta}\). Then \(xD^{-\eta}=D\),
\(D^{1-\eta}\leq D\), and \(x^{1-\eta}\leq D\), so
\begin{equation}\label{eq:balanced-hyperbola}
\begin{aligned}
 |\psi(x)-x|\leq x^\delta\Big[
 &\tfrac32+K+\frac{K}{\eta^2}\\
 &+(1-\delta)\Big(\tfrac12+\frac K2+\frac{2K}{\eta}\Big)\log x
 \Big].
\end{aligned}
\end{equation}
The split balances the powers in the first and third terms of
\eqref{eq:effective-hyperbola}. The unbalanced choice below uses the
freedom to reduce the term with the larger coefficient over the range
needed for Ramanujan's inequality.

\begin{cor}\label{cor:explicit-theta}
For every real $x\geq e^{400}$,
\begin{equation}\label{eq:explicit-theta-mertens}
 |\theta(x)-x|
 <x\log x\left(0.04\,x^{-0.07}+18300\,x^{-0.087}\right).
\end{equation}
\end{cor}
\begin{proof}
Use \Cref{thm:explicit-mertens,prop:mertens-transfer} with
$K=12000$ and $\eta=22/235$. With a split at $D=x^\delta$, the
first and third terms in \eqref{eq:effective-hyperbola} have powers
$x^\delta$ and $x^{1-\eta\delta}$, respectively, apart from logarithmic
factors. These powers agree at $\delta=235/257$. Taking the slightly
larger value $\delta=0.93$ reduces the third term, which carries the
factor $2K/\eta$. We use $D=x^{0.93}$ and $N=x^{0.07}$ below.
The split lies in the range of the Mertens bound because
$\log D\geq372$. In \eqref{eq:effective-hyperbola}, the first term
is at most $0.03875 x^{0.93}\log x$ for $\log x\geq400$.
The ratio of the second term to $(2Kx/\eta) D^{-\eta}\log N$
is
\[
 \frac{\eta}{2}x^{-0.07}\left(\tfrac12+\frac{1}{\log N}\right)
 <10^{-6}.
\]
The ratio of the last term to the same expression is
$x^{-0.07\eta}/(0.14\eta\log x)<0.015$ for $\log x\geq400$.
Indeed, $0.07\eta\cdot400>2.6$, $e^{2.6}>13$, and
$0.14\eta\cdot400=56\eta$. Since $2K(0.07)/\eta<17946$ and
$0.93\eta>0.087$, the second, third, and fourth terms together
contribute less than $18300x^{1-0.087}\log x$.
Finally, $0\leq\psi(x)-\theta(x)\leq\sqrt{x}\log x$ by \Cref{thm2}.
The first term and this correction contribute less than
$0.04x^{0.93}\log x$, since $x^{-0.43}<0.00125$ in the stated range.
\end{proof}

\section{The Range of Ramanujan's Inequality}\label{sec:range}
\subsection{\texorpdfstring{Monotonicity of $\mathcal{G}$}{Monotonicity of G}}
The step discontinuities of the prime-counting function determine the behavior of $\mathcal{G}$ near primes.

\begin{prop}\label{prop4}
The function $\mathcal{G}$ is strictly decreasing between consecutive points of discontinuity for $x>2e$ and has upward jumps at arbitrarily large $x$.
\end{prop}
\begin{proof}
Both $\pi(x)$ and $\pi(x/e)$ are constant on intervals containing neither a prime $p$ nor a point $ep$ with $p$ prime. On these intervals with $x>2e$,
\[
 \mathcal{G}'(x)=-e\,\pi(x/e)\frac{\log x-1}{(\log x)^2}<0.
\]
At any prime $p$, the point $p/e$ is not an integer, so $\pi(x/e)$ is continuous at $p$. Here we use the irrationality of $e$. Because $\pi$ is right-continuous and jumps by one at $p$,
\begin{equation}\label{eq:G-jump}
 \mathcal{G}(p)-\lim_{x\uparrow p}\mathcal{G}(x)
 =\pi(p)^2-(\pi(p)-1)^2=2\pi(p)-1>0.
\end{equation}
There are arbitrarily large primes, so $\mathcal{G}$ has upward jumps in every tail of the real line.
\end{proof}

For $x>1$ and $h>0$, \eqref{eq:unweighted} gives
\[
 \sum_{n\leq x+h}M((x+h)/n)-\sum_{n\leq x}M(x/n)=0.
\]
The difference remains zero after multiplication by $-2x/(\log x)^2$.
This rules out the strictly negative upper bound
$-x(\sqrt{x+h}-\sqrt{x})/(\log x)^2$ on the scaled difference used in
the earlier version. The example in \Cref{sec:mertens} also disproves
von Sterneck's conjecture.

\subsection{An explicit range from the Mertens estimate}\label{sec:explicit}

The asymptotic expansion in \eqref{21} identifies the scale on which
the sign is decided. The following argument controls the error on that
scale using \Cref{cor:explicit-theta}.

\begin{lem}\label{lem:explicit-pi-mertens}
For every real $y\geq e^{489}$,
\begin{equation}\label{eq:explicit-pi-mertens}
 |\pi(y)-\Li(y)|<0.1\,y^{0.93}+21000\,y^{0.913}.
\end{equation}
\end{lem}
\begin{proof}
Abel summation and integration by parts give, for $y\geq2$,
\[
 \pi(y)-\Li(y)
 =\frac{\theta(y)-y}{\log y}
   +\int_2^y\frac{\theta(t)-t}{t(\log t)^2}\,dt
   +\frac{2}{\log 2}-\Li(2).
\]
The final constant has absolute value less than $10$. Indeed, the
series for $\Li(e)$ gives $0<\Li(e)<3$, and
$\Li(e)-\Li(2)\leq(e-2)/\log 2<3/2$. On $2\leq t\leq e^{400}$,
the trivial bound $0\leq\theta(t)\leq t\log t$ gives
$|\theta(t)-t|\leq400t$. Since $\log 2>2/3$, this part of the integral
together with the constant contributes less than $1000e^{400}$.
For \(y\geq e^{489}\) this is less than \(0.01\,y^{0.93}\), since
\(0.93\cdot489-400=54.77\) and \(e^{54.77}>10^5\). The exponential
comparison follows from the nonnegative Taylor sum through degree six.
For the remaining integral use \eqref{eq:explicit-theta-mertens} and
$1/\log t\leq1/400$. The resulting estimate is
\[
 \begin{aligned}
 |\pi(y)-\Li(y)|<{}&
 \left[0.01+0.04\left(1+\frac{1}{400\cdot0.93}\right)\right]y^{0.93}\\
 &+18300\left(1+\frac{1}{400\cdot0.913}\right)y^{0.913}.
 \end{aligned}
\]
Bounding the two coefficients gives \eqref{eq:explicit-pi-mertens}.
\end{proof}

\begin{lem}\label{lem:li-margin}
Put $A(v)=v e^{-v}\Li(e^v)$ for $v>0$. For $L\geq300$,
\begin{equation}\label{eq:li-margin}
 0<A(L)<1.1,\qquad
 \frac{L}{L-1}A(L-1)-A(L)^2
 \geq L^{-4}+14L^{-5}-3400L^{-6}>0.95L^{-4}.
\end{equation}
\end{lem}
\begin{proof}
Write $P(u)=\sum_{k=0}^5 k!u^k$ and split the integral for $\Li(e^v)$
at $e^{v/2}$. Six integrations by parts give
\[
 A(v)-P(1/v)
 =2e^{-v/2}\bigl(A(v/2)-P(2/v)\bigr)
   +720v e^{-v}\int_{v/2}^v \frac{e^t}{t^7}\,dt .
\]
For $w\geq51$, splitting $\Li(e^w)=\Li(e)+\int_1^w e^t/t\,dt$
at $w/2$ gives
\[
 0<A(w)\leq 2+w e^{-w/2}+3w e^{-w}<3 .
\]
Thus the first term is bounded in absolute value by $8e^{-v/2}$
when $v\geq102$, since $0<P(2/v)<2$. The derivative of $e^t/t^7$
is $(1-7/t)e^t/t^7$. On $v/2\leq t\leq v$ the factor $1-7/t$
is at least $1-14/v$, so
\begin{equation}\label{eq:li-polynomial-error}
 |A(v)-P(1/v)|
 \leq 8e^{-v/2}+\frac{720v^{-6}}{1-14/v}<900v^{-6}
 \qquad(v\geq102).
\end{equation}
The final step follows from $8v^6 e^{-v/2}<1$ at $v=102$ and
the decrease of $v^6e^{-v/2}$ on this range.

Set $u=1/L$ and $Q(u)=(1-u)^{-1}P(u/(1-u))$.
In the binomial expansion of $Q(u)$, the coefficient of $u^j$ is
$\sum_{k=0}^{\min(j,5)} k!\binom{j}{k}$. For $j=0,\ldots,5$ these
coefficients are $1,2,5,16,65,326$. All higher coefficients are
nonnegative. The coefficients of $P(u)^2$ through degree five
are $1,2,5,16,64,312$. Its coefficients of degrees six through ten
are $372,768,2016,5760,14400$. Their contribution is less than
$400u^6$ for $0<u\leq1/103$. Hence
\[
 Q(u)-P(u)^2\geq u^4+14u^5-400u^6\qquad(0<u\leq1/103).
\]
In replacing $Q(u)$ by $L A(L-1)/(L-1)$, the error is at most
$900u^6/(1-u)^7<1000u^6$ by \eqref{eq:li-polynomial-error}.
Both $A(L)$ and $P(u)$ lie in $(0,1.1)$ in this range, so
$|A(L)^2-P(u)^2|<1980u^6$. Combining the three errors proves the
first lower bound in \eqref{eq:li-margin}. The last comparison uses
$3400/L^2<0.05$ for $L\geq300$.
\end{proof}

\begin{thm}[Ramanujan's inequality by an explicit Mertens estimate]
\label{thm:explicit-ramanujan}
For every real \(x\geq e^{490}\),
\begin{equation}\label{eq:explicit-ramanujan-mertens}
 \pi(x)^2 < \frac{ex}{\log x}\,\pi(x/e).
\end{equation}
\end{thm}
\begin{proof}
Let \(L=\log x\) and \(B(v)=v e^{-v}\pi(e^v)\) for \(v>0\).
Normalizing the terms at \(x\) and \(x/e\) gives
\[
 \frac{L^2}{x^2}\mathcal G(x)=B(L)^2-\frac{L}{L-1}B(L-1).
\]
Since \(L-1\geq489\), \cref{lem:explicit-pi-mertens} applies at
both \(e^L\) and \(e^{L-1}\). For \(v\in\{L,L-1\}\), use
\(v\leq L\) in the prefactor and \(v\geq L-1\) in the
exponential terms to obtain \(|B(v)-A(v)|<\epsilon(L)\), where
\[
 \epsilon(L)=L\left(0.1e^{-0.07(L-1)}+21000e^{-0.087(L-1)}\right).
\]
The logarithmic derivatives of the two terms of \(L^4\epsilon(L)\)
are \(5/L-0.07\) and \(5/L-0.087\), both negative for \(L\geq490\).
At \(L=490\), the Taylor sums for the exponential through degree
\(60\) give \(e^{34.23}>7\cdot10^{14}\) and
\(e^{42.543}>2.9\cdot10^{18}\). Hence
\begin{equation}\label{eq:490-error-budget}
 L^4\epsilon(L)\leq 490^4\epsilon(490)
 <490^5\left(\frac{0.1}{7\cdot10^{14}}+
             \frac{21000}{2.9\cdot10^{18}}\right)<0.21.
\end{equation}
Using \(0\leq B(L)<A(L)+\epsilon(L)\) and
\(B(L-1)>A(L-1)-\epsilon(L)\) in the normalized gap, we obtain
\[
 \frac{L^2}{x^2}\mathcal G(x)
 \leq -\left(\frac{L}{L-1}A(L-1)-A(L)^2\right)
      +\left(2A(L)+\frac{L}{L-1}\right)\epsilon(L)+\epsilon(L)^2.
\]
The first parenthesis exceeds \(0.95L^{-4}\) by
\cref{lem:li-margin}. Also \(2A(L)+L/(L-1)<4\), and
\(\epsilon(L)^2<0.21^2L^{-8}\leq0.21^2L^{-4}\). The right-hand
side is therefore strictly less than
\[
 \big(-0.95+4(0.21)+0.21^2\big)L^{-4}=-0.0659L^{-4}.
\]
This proves \eqref{eq:explicit-ramanujan-mertens} on the full real
interval \(x\geq e^{490}\).
\end{proof}

\subsection{Numerical values of the logarithmic Ramanujan gap}\label{sec:numerical}
By \cref{thm:explicit-ramanujan}, \(-\mathcal G(e^t)>0\) for
\(t\geq490\). We compute
\[
 H(t)=\log\bigl(-\mathcal G(e^t)\bigr)
\]
at \(t=490,500,\ldots,3300\). \Cref{tab:grid} gives all 282 values
rounded to eight decimal places, beginning with \(t=490\) at \(x=e^{490}\).
\Cref{fig1} plots these values against \(t=\log x\).

Subtracting \(2t-6\log t\) resolves the next term. Write
\(B(v)=v e^{-v}\pi(e^v)\). The coefficients computed in the proof of
\cref{lem:li-margin} and the exponentially decaying prime-counting error
give
\[
 \frac{t}{t-1} B(t-1)-B(t)^2
    =t^{-4}+14t^{-5}+O(t^{-6}).
\]
Taking logarithms gives \(R(t)=H(t)-2t+6\log t=14/t+O(t^{-2})\).
The inset in \cref{fig1} plots \(R(t)\), which is about \(0.02877\)
at \(t=490\) and \(0.00425\) at \(t=3300\).

\begin{figure}[!htbp]
 \centering
 \includegraphics[width=0.96\linewidth]{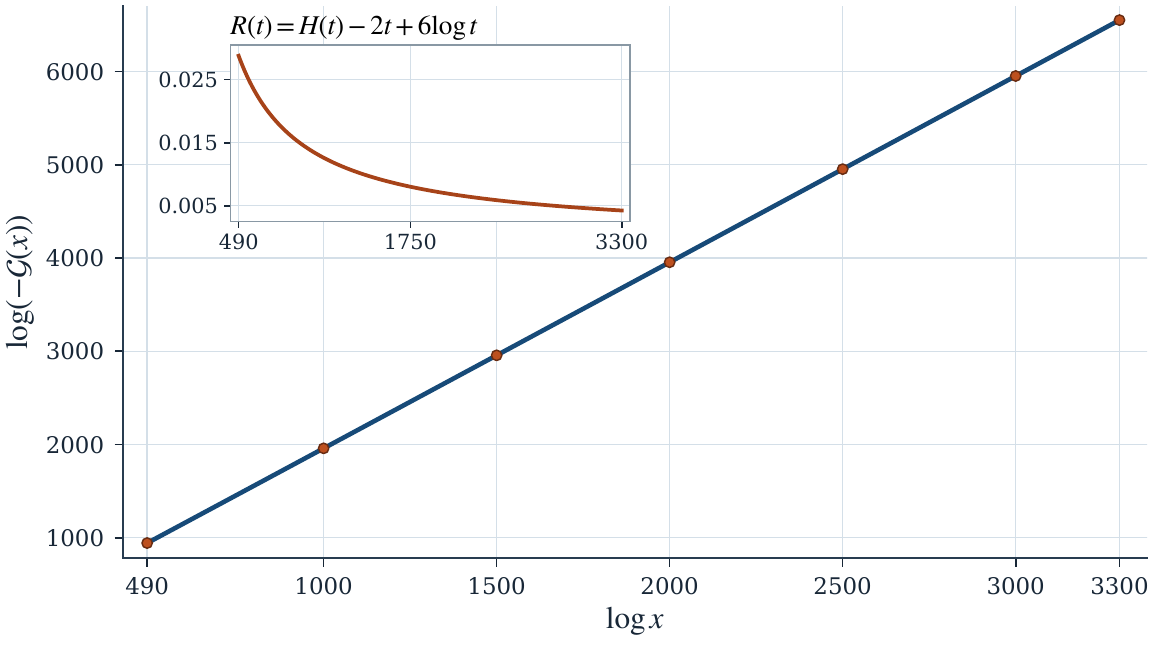}
 \makeatletter
 \long\def\@figurecaption#1#2{%
   {\centering\figurecaptionfont\bfseries\boldmath #1\enspace #2\par}%
 }
 \makeatother
 \caption{Computed values of \(H(t)=\log(-\mathcal{G}(e^t))\) at
 \(t=490,500,\ldots,3300\). The inset shows
 \(R(t)=H(t)-2t+6\log t\) on the same grid.}
 \label{fig1}
\end{figure}

To certify the rounding, write \(v=\log y\) and use \(A(v)\) as in
\cref{lem:li-margin}.
We first obtain a bound sharper than \eqref{eq:explicit-pi-mertens}
for this calculation. Let $N(U)$ count nontrivial zeros of $\zeta$
with $0<\gamma\leq U$, including multiplicities.
The estimate of Hasanalizade, Shen, and Wong
\cite[Corollary~1.2]{hasanalizade-shen-wong} gives
$N(U)\leq U\log U$ for $U\geq14$. Indeed, its main term is at most
$(U/6)\log U$ and its error is below
$0.3611\log U+9.3675$. Since the first positive ordinate exceeds $14$
\cite[Example~1]{brent-platt-trudgian}, partial summation gives
\[
 \sum_{|\gamma|\leq T}\frac{1}{|\rho|}
 \leq 2\left(\frac{N(T)}{T}+\int_{14}^{T}\frac{N(u)}{u^2}\,du\right)
 \leq 2\log T+(\log T)^2 \qquad(T\geq14).
\]
For $v\geq40$, take $T=y^{1/4}$. This satisfies
$\max\{51,v\}<T<(\sqrt y-2)/2$. The explicit formula of
Cully-Hugill and Johnston \cite[Theorem~1.2 and Table~4]{cully-hugill-johnston}
bounds the truncation error by $7y\log y/T$. Combining the bound on
the zero sum with $\Re\rho\leq7/8$ and
$\psi(y)-\theta(y)\leq\sqrt y\log y$ gives
\[
 \frac{|\theta(y)-y|}{y^{7/8}v^2}
 \leq \frac{1}{16}+\frac{1}{2v}+
          \frac{7e^{-v/8}+e^{-3v/8}}{v}<0.1 .
\]
Use Abel summation as in \Cref{lem:explicit-pi-mertens}, splitting at
$e^{40}$. On $2\leq u\leq e^{40}$, the bound
$|\theta(u)-u|\leq40u$ gives a contribution below $100e^{40}$,
including the constant of integration. On the remaining interval
the integrand is bounded by $0.1u^{-1/8}$. Consequently,
\[
 |\pi(y)-\Li(y)|
 \leq 0.1v y^{7/8}+\tfrac{0.8}{7}y^{7/8}+100e^{40}
 < v y^{7/8} \qquad(v\geq489).
\]
Thus $|B(v)-A(v)|<v^2e^{-v/8}$ at both arguments needed for each row.

For integer $v\geq489$, put $P(v)=\sum_{k=0}^{15}k!/v^k$.
The integration by parts used in \eqref{eq:li-polynomial-error}, now
with sixteen terms, gives
\[
 P(v)-8e^{-v/2}\leq A(v)
 \leq P(v)+8e^{-v/2}+\frac{16!}{v^{16}(1-34/v)}.
\]
Here the endpoint at $v/2$ has absolute value below $8e^{-v/2}$
and the integral remainder is nonnegative. Its upper bound uses
$(e^u/u^{17})'=(1-17/u)e^u/u^{17}$ on $v/2\leq u\leq v$.
For rational bounds we majorize $e^{-v/2}$ by $2^{-\lfloor v/2\rfloor}$
and $e^{-v/8}$ by $(3072/3481)^v$, using the degree-three Taylor
polynomial for $e^{1/8}$.

With these enclosures for $B(t)$ and $B(t-1)$, interval arithmetic gives
positive bounds on
\[
 J(t)=\frac{t}{t-1}B(t-1)-B(t)^2,\qquad
 H(t)=2t-2\log t+\log J(t).
\]
To evaluate the logarithms, we rescale their arguments by powers
of two and use the series
$\log z=2\sum_{j\geq0}q^{2j+1}/(2j+1)$ with
$q=(z-1)/(z+1)$ for $1\leq z\leq2$. After $m$ terms its remainder
is bounded by $2q^{2m+1}/((2m+1)(1-q^2))$. With $m=30$, each interval
for $H(t)$ has width below $3\times10^{-10}$ and lies inside a single
rounding interval at eight decimal places. The source archive
includes the script that reproduces the calculation.

\begin{sidewaystable}[p]
 \let\raggedright\centering
 \caption{Computed values of $H(t)=\log(-\mathcal G(e^t))$ for
 $t=490,500,\ldots,3300$, rounded off to eight decimal places.}
 \label{tab:grid}
 \fontsize{8}{8.7}\selectfont
 \setlength{\tabcolsep}{5pt}
 \setlength{\RamanujanValueWidth}{\dimexpr(\textheight-24\tabcolsep-0.1pt)/6-30pt\relax}
 \newcolumntype{Z}{@{\extracolsep{\fill}\hspace{\tabcolsep}}*{6}{W{c}{30pt}W{c}{\RamanujanValueWidth}}}
 \begin{tabular*}{\textheight}{Z}
 \toprule
 \multicolumn{12}{@{}c@{}}{%
   \setlength{\fboxsep}{0pt}%
   \colorbox{RamanujanHeader}{%
     \begin{tabular*}{\textheight}{Z}
       \RamanujanHeading{t} & \RamanujanHeading{H(t)} &
       \RamanujanHeading{t} & \RamanujanHeading{H(t)} &
       \RamanujanHeading{t} & \RamanujanHeading{H(t)} &
       \RamanujanHeading{t} & \RamanujanHeading{H(t)} &
       \RamanujanHeading{t} & \RamanujanHeading{H(t)} &
       \RamanujanHeading{t} & \RamanujanHeading{H(t)}
     \end{tabular*}}%
 } \\
 \midrule
\rowcolor{RamanujanBand} 490 & 942.86233740 & 960 & 1878.81303496 & 1430 & 2816.41723496 & 1900 & 3754.70972649 & 2370 & 4693.38204416 & 2840 & 5632.29557943 \\
500 & 962.74054182 & 970 & 1898.75070684 & 1440 & 2836.37535463 & 1910 & 3774.67819164 & 2380 & 4713.35675607 & 2850 & 5652.27447244 \\
\rowcolor{RamanujanBand} 510 & 982.62116960 & 980 & 1918.68901954 & 1450 & 2856.33376461 & 1920 & 3794.64682167 & 2390 & 4733.33157412 & 2860 & 5672.25343944 \\
520 & 1002.50412617 & 990 & 1938.62796002 & 1460 & 2876.29246090 & 1930 & 3814.61561485 & 2400 & 4753.30649742 & 2870 & 5692.23247991 \\
\rowcolor{RamanujanBand} 530 & 1022.38932240 & 1000 & 1958.56751563 & 1470 & 2896.25143958 & 1940 & 3834.58456951 & 2410 & 4773.28152508 & 2880 & 5712.21159335 \\
540 & 1042.27667418 & 1010 & 1978.50767409 & 1480 & 2916.21069681 & 1950 & 3854.55368398 & 2420 & 4793.25665626 & 2890 & 5732.19077924 \\
\rowcolor{RamanujanBand} 550 & 1062.16610206 & 1020 & 1998.44842351 & 1490 & 2936.17022884 & 1960 & 3874.52295662 & 2430 & 4813.23189008 & 2900 & 5752.17003709 \\
560 & 1082.05753089 & 1030 & 2018.38975232 & 1500 & 2956.13003199 & 1970 & 3894.49238582 & 2440 & 4833.20722571 & 2910 & 5772.14936640 \\
\rowcolor{RamanujanBand} 570 & 1101.95088952 & 1040 & 2038.33164931 & 1510 & 2976.09010264 & 1980 & 3914.46196999 & 2450 & 4853.18266231 & 2920 & 5792.12876667 \\
580 & 1121.84611057 & 1050 & 2058.27410356 & 1520 & 2996.05043727 & 1990 & 3934.43170758 & 2460 & 4873.15819907 & 2930 & 5812.10823743 \\
\rowcolor{RamanujanBand} 590 & 1141.74313008 & 1060 & 2078.21710451 & 1530 & 3016.01103240 & 2000 & 3954.40159703 & 2470 & 4893.13383516 & 2940 & 5832.08777819 \\
600 & 1161.64188735 & 1070 & 2098.16064184 & 1540 & 3035.97188464 & 2010 & 3974.37163684 & 2480 & 4913.10956979 & 2950 & 5852.06738848 \\
\rowcolor{RamanujanBand} 610 & 1181.54232471 & 1080 & 2118.10470557 & 1550 & 3055.93299064 & 2020 & 3994.34182551 & 2490 & 4933.08540215 & 2960 & 5872.04706782 \\
620 & 1201.44438731 & 1090 & 2138.04928598 & 1560 & 3075.89434715 & 2030 & 4014.31216156 & 2500 & 4953.06133147 & 2970 & 5892.02681575 \\
\rowcolor{RamanujanBand} 630 & 1221.34802293 & 1100 & 2157.99437359 & 1570 & 3095.85595095 & 2040 & 4034.28264355 & 2510 & 4973.03735698 & 2980 & 5912.00663181 \\
640 & 1241.25318183 & 1110 & 2177.93995921 & 1580 & 3115.81779891 & 2050 & 4054.25327005 & 2520 & 4993.01347789 & 2990 & 5931.98651554 \\
\rowcolor{RamanujanBand} 650 & 1261.15981662 & 1120 & 2197.88603390 & 1590 & 3135.77988793 & 2060 & 4074.22403966 & 2530 & 5012.98969347 & 3000 & 5951.96646649 \\
660 & 1281.06788205 & 1130 & 2217.83258893 & 1600 & 3155.74221498 & 2070 & 4094.19495097 & 2540 & 5032.96600295 & 3010 & 5971.94648421 \\
\rowcolor{RamanujanBand} 670 & 1300.97733495 & 1140 & 2237.77961582 & 1610 & 3175.70477711 & 2080 & 4114.16600263 & 2550 & 5052.94240561 & 3020 & 5991.92656826 \\
680 & 1320.88813405 & 1150 & 2257.72710632 & 1620 & 3195.66757138 & 2090 & 4134.13719329 & 2560 & 5072.91890071 & 3030 & 6011.90671820 \\
\rowcolor{RamanujanBand} 690 & 1340.80023990 & 1160 & 2277.67505238 & 1630 & 3215.63059495 & 2100 & 4154.10852161 & 2570 & 5092.89548754 & 3040 & 6031.88693359 \\
700 & 1360.71361478 & 1170 & 2297.62344617 & 1640 & 3235.59384500 & 2110 & 4174.07998630 & 2580 & 5112.87216537 & 3050 & 6051.86721401 \\
\rowcolor{RamanujanBand} 710 & 1380.62822256 & 1180 & 2317.57228004 & 1650 & 3255.55731878 & 2120 & 4194.05158606 & 2590 & 5132.84893350 & 3060 & 6071.84755902 \\
720 & 1400.54402864 & 1190 & 2337.52154656 & 1660 & 3275.52101357 & 2130 & 4214.02331961 & 2600 & 5152.82579124 & 3070 & 6091.82796821 \\
\rowcolor{RamanujanBand} 730 & 1420.46099985 & 1200 & 2357.47123846 & 1670 & 3295.48492672 & 2140 & 4233.99518570 & 2610 & 5172.80273790 & 3080 & 6111.80844116 \\
740 & 1440.37910437 & 1210 & 2377.42134869 & 1680 & 3315.44905561 & 2150 & 4253.96718310 & 2620 & 5192.77977279 & 3090 & 6131.78897746 \\
\rowcolor{RamanujanBand} 750 & 1460.29831169 & 1220 & 2397.37187033 & 1690 & 3335.41339769 & 2160 & 4273.93931058 & 2630 & 5212.75689525 & 3100 & 6151.76957669 \\
760 & 1480.21859251 & 1230 & 2417.32279665 & 1700 & 3355.37795044 & 2170 & 4293.91156694 & 2640 & 5232.73410461 & 3110 & 6171.75023845 \\
\rowcolor{RamanujanBand} 770 & 1500.13991866 & 1240 & 2437.27412109 & 1710 & 3375.34271137 & 2180 & 4313.88395100 & 2650 & 5252.71140021 & 3120 & 6191.73096233 \\
780 & 1520.06226308 & 1250 & 2457.22583725 & 1720 & 3395.30767806 & 2190 & 4333.85646159 & 2660 & 5272.68878141 & 3130 & 6211.71174795 \\
\rowcolor{RamanujanBand} 790 & 1539.98559976 & 1260 & 2477.17793887 & 1730 & 3415.27284812 & 2200 & 4353.82909754 & 2670 & 5292.66624755 & 3140 & 6231.69259490 \\
800 & 1559.90990366 & 1270 & 2497.13041983 & 1740 & 3435.23821920 & 2210 & 4373.80185773 & 2680 & 5312.64379800 & 3150 & 6251.67350280 \\
\rowcolor{RamanujanBand} 810 & 1579.83515066 & 1280 & 2517.08327419 & 1750 & 3455.20378900 & 2220 & 4393.77474102 & 2690 & 5332.62143214 & 3160 & 6271.65447126 \\
820 & 1599.76131757 & 1290 & 2537.03649611 & 1760 & 3475.16955524 & 2230 & 4413.74774632 & 2700 & 5352.59914934 & 3170 & 6291.63549989 \\
\rowcolor{RamanujanBand} 830 & 1619.68838201 & 1300 & 2556.99007992 & 1770 & 3495.13551571 & 2240 & 4433.72087252 & 2710 & 5372.57694898 & 3180 & 6311.61658832 \\
840 & 1639.61632243 & 1310 & 2576.94402005 & 1780 & 3515.10166820 & 2250 & 4453.69411856 & 2720 & 5392.55483047 & 3190 & 6331.59773617 \\
\rowcolor{RamanujanBand} 850 & 1659.54511803 & 1320 & 2596.89831107 & 1790 & 3535.06801056 & 2260 & 4473.66748336 & 2730 & 5412.53279320 & 3200 & 6351.57894307 \\
860 & 1679.47474875 & 1330 & 2616.85294769 & 1800 & 3555.03454067 & 2270 & 4493.64096588 & 2740 & 5432.51083657 & 3210 & 6371.56020864 \\
\rowcolor{RamanujanBand} 870 & 1699.40519523 & 1340 & 2636.80792470 & 1810 & 3575.00125646 & 2280 & 4513.61456509 & 2750 & 5452.48895999 & 3220 & 6391.54153254 \\
880 & 1719.33643878 & 1350 & 2656.76323705 & 1820 & 3594.96815587 & 2290 & 4533.58827995 & 2760 & 5472.46716290 & 3230 & 6411.52291438 \\
\rowcolor{RamanujanBand} 890 & 1739.26846132 & 1360 & 2676.71887977 & 1830 & 3614.93523688 & 2300 & 4553.56210945 & 2770 & 5492.44544470 & 3240 & 6431.50435382 \\
900 & 1759.20124541 & 1370 & 2696.67484802 & 1840 & 3634.90249752 & 2310 & 4573.53605262 & 2780 & 5512.42380483 & 3250 & 6451.48585050 \\
\rowcolor{RamanujanBand} 910 & 1779.13477417 & 1380 & 2716.63113705 & 1850 & 3654.86993584 & 2320 & 4593.51010845 & 2790 & 5532.40224272 & 3260 & 6471.46740406 \\
920 & 1799.06903128 & 1390 & 2736.58774222 & 1860 & 3674.83754991 & 2330 & 4613.48427599 & 2800 & 5552.38075783 & 3270 & 6491.44901417 \\
\rowcolor{RamanujanBand} 930 & 1819.00400096 & 1400 & 2756.54465899 & 1870 & 3694.80533786 & 2340 & 4633.45855426 & 2810 & 5572.35934960 & 3280 & 6511.43068046 \\
940 & 1838.93966792 & 1410 & 2776.50188292 & 1880 & 3714.77329782 & 2350 & 4653.43294234 & 2820 & 5592.33801748 & 3290 & 6531.41240261 \\
\rowcolor{RamanujanBand} 950 & 1858.87601736 & 1420 & 2796.45940966 & 1890 & 3734.74142796 & 2360 & 4673.40743928 & 2830 & 5612.31676093 & 3300 & 6551.39418026 \\
 \botrule
 \end{tabular*}
\end{sidewaystable}

\section{Future Research Prospects}
In Proposition \ref{prop:mertens-transfer}, the weighted sum and the integral
in \eqref{eq:mertens-sum-integral} were bounded separately after the
boundary cancellation. Estimating their difference directly could
reduce the error from the larger arguments of $M$ and lower the
threshold obtained by this method.

One direction is to replace $e$ by $\alpha$ in \eqref{24}. Hassani \cite[Theorem~1.3]{22} showed that the inequality reverses for every fixed $0<\alpha<e$ once $x$ is sufficiently large, while it holds in its original direction for fixed $\alpha\geq e$ and sufficiently large $x$. Hassani also studied higher powers of $\pi(x)$ by iterating Ramanujan's inequality \cite[Theorem~1]{28}. For each positive integer $n$ and sufficiently large $x$, this gives
\[
 \pi(x)^{2^n}<
 \frac{e^n}{\prod_{k=1}^{n}\left(1-\frac{k-1}{\log x}\right)^{2^{n-k}}}
 \left(\frac{x}{\log x}\right)^{2^n-1}\pi\left(\frac{x}{e^n}\right).
\]

Another direction is to study polynomial expressions in $\pi(x)$ and its rescaled arguments. For example, the preprint \cite[Section~3]{29} considers
\[
 \frac{3ex}{\log x}\left(\pi\left(\frac{x}{e}\right)\right)^{3^n-1}
 <\pi(x)^{3^n}
  +\frac{3e^2x}{(\log x)^2}\left(\pi\left(\frac{x}{e^2}\right)\right)^{3^n-2}.
\]
For each fixed $n>1$, the prime number theorem shows that the ratio of the left-hand side to the first term on the right tends to $3e^{2-3^n}<1$. For $n=1$ the ratio tends to $3/e>1$ and the second term on the right is $O(x^2/\log^3x)$, which gives the reversed inequality for sufficiently large $x$.

\section*{Acknowledgments}
I am grateful to Prof. Adrian W. Dudek (Adjunct Associate Professor, Department of Mathematics and Physics, University of Queensland, Australia) for inspiring this work. His publications helped me understand the main ideas behind the problem.

I also thank Prof. Tadej Kotnik (Faculty of Electrical Engineering, University of Ljubljana, Slovenia) for helping me understand several results and derivations concerning the Mertens function.

\section*{Statements and Declarations}
\subsection*{Conflicts of Interest Statement}
The author declares no conflicts of interest.
\subsection*{Data Availability Statement}
 I as the sole author of this article confirm that the data supporting the findings of this study are available within the article [and/or] its supplementary materials. The numerical values appear in Table \ref{tab:grid}.

\clearpage
\setlength{\bibsep}{0.8em}

\clearpage
\begin{appendices}
\section{Mellin continuation in the zero-free theorem}\label{app:continuation}
\setcounter{equation}{0}
\renewcommand{\theequation}{\thesection.\arabic{equation}}
\renewcommand{\theHequation}{appendix.\thesection.\arabic{equation}}
We give the analytic deduction used in the preprint
\cite[Proposition~2.1]{qriemann}. Let \(\beta_*\) be the supremum
of \(1/2\) and the real parts of the zeros of the primitive
finite-order Hecke \(L\)-functions over \(\mathbb Q(\sqrt{-3})\).
Absolute convergence of the Euler products gives \(\beta_*\leq1\).
Suppose \(\beta_*>b_0=7/8\), and put \(c=-11/16\). Under this
supposition the preprint's estimates
\cite[Eq.~(10.2), Proposition~20.3 and Section~20.6]{qriemann}
provide the following data for each primitive character \(\nu\).
Write \(L_\nu^S(s)\) for its \(L\)-function with a finite set \(S\) of
Euler factors removed. There is a holomorphic \(Q_\nu\) on
\(\Re s>b_0\) with \(|Q_\nu(s)-1|\leq1/2\). In terms of
\[
 W_\nu(s)=\frac{e^{(s-5/6)^2}Q_\nu(s)}{L_\nu^S(s)},\qquad
 f_\nu(Z)=\frac{1}{2\pi i}\int_{\Re s=2} Z^{s+c}W_\nu(s)\,ds,
\]
the estimates give a function \(J_\nu(Z)\) satisfying, for all
sufficiently large \(Z\),
\begin{equation}\label{eq:mellin-two-estimates}
 |J_\nu(Z)|\ll_\nu Z^{b_0+c+\omega},
 \qquad
 |J_\nu(Z)-f_\nu(Z)|\ll_\nu Z^{\beta_*+c-\sigma}.
\end{equation}
Here \(0<\omega<\beta_*-b_0\) and \(\sigma>0\) are independent of
\(\nu\). These are the two bounds on the same sum that come from
the reflected and Poisson expansions in the preprint.

Set \(m=\min\{\beta_*-b_0-\omega,\sigma\}>0\).
The triangle inequality in \eqref{eq:mellin-two-estimates} gives
\(f_\nu(Z)\ll_\nu Z^{\beta_*+c-m}\) as \(Z\to\infty\).
For \(Z\downarrow0\), shift the line of integration to any \(B>2\).
In the strip \(2\leq\Re s\leq B\), the reciprocal Euler product
and \(Q_\nu\) are bounded and the Gaussian factor decays as
\(e^{-(\Im s)^2}\) up to a constant. The horizontal contour
integrals vanish, giving \(f_\nu(Z)=O_{\nu,B}(Z^{B+c})\).
Consequently
\[
 F_\nu(s)=\int_0^\infty f_\nu(Z)Z^{-s-c}\,\frac{dZ}{Z}
\]
converges locally uniformly and is holomorphic on
\(\Re s>\beta_*-m\).

Evaluate the inverse Mellin integral at \(Z=e^u\).
Then \(f_\nu(e^u)e^{-(2+c)u}\) is the inverse Fourier transform of
\(t\mapsto W_\nu(2+it)\). Both functions are integrable by the
bounds just obtained and the Gaussian decay. Fourier inversion
therefore gives \(F_\nu(2+it)=W_\nu(2+it)\). The identity theorem
extends this equality to \(\Re s>1\). The quotient
\(e^{-(s-5/6)^2}F_\nu(s)/Q_\nu(s)\) extends \(1/L_\nu^S(s)\)
holomorphically to \(\Re s>\beta_*-m\), since
\(\beta_*-m>b_0\) and \(|Q_\nu|\geq1/2\).
For \(\Re s>0\), the absolute value of
\(\nu(\mathfrak p)(N\mathfrak p)^{-s}\) is less than one. Hence each
deleted factor \(1-\nu(\mathfrak p)(N\mathfrak p)^{-s}\) is nonzero
there. By the definition of \(\beta_*\), some
member of the family has a zero with real part above
\(\beta_*-m\). Its reciprocal would have a pole there, contradicting
the extension. This proves the primitive Hecke conclusion from
\eqref{eq:mellin-two-estimates}. The Euler factor argument in
\cref{thm:qriemann} then gives the assertion for \(\zeta\).
\end{appendices}

\end{document}